\documentclass[11pt]{amsart}
\usepackage{amssymb, latexsym, amsmath, eucal}
\usepackage[utf8]{inputenc}
\usepackage[english]{babel}
\usepackage{csquotes}
\usepackage{tikz}
\usepackage{graphicx}
\usepackage{subcaption}
\usepackage{setspace}
\usepackage{enumitem}

\usepackage{thmtools}
\usepackage{thm-restate}
\usepackage{hyperref}
\usepackage{cleveref}

\usepackage[style = alphabetic]{biblatex}
\usepackage{hyperref}

\usepackage{todonotes}

\theoremstyle{plain}
\newtheorem{thm}{Theorem}[section]

\newtheorem{lemma}[thm]{Lemma}
\newtheorem{conj}[thm]{Conjecture}
\newtheorem{cor}[thm]{Corollary}

\theoremstyle{definition}
\newtheorem{defn}[thm]{Definition}

\newtheorem{remark}[thm]{Remark}

\title{Property (T) in k-gonal random groups}
\author{MurphyKate Montee}
\date{March 2021}

\address{Department of Mathematics and Statistics, Carleton College}

\email{murphykate.montee@gmail.com}

\begin{document}

\begin{abstract}
    The $k$-gonal models of random groups are defined as the quotients of free groups on $n$ generators by cyclically reduced words of length $k$. As $k$ tends to infinity, this model approaches the Gromov density model. In this paper we show that for any fixed $d_0 \in (0, 1)$, if positive $k$-gonal random groups satisfy Property (T) with overwhelming probability for densities $d >d_0$, then so do $nk$-gonal random groups, for any $n \in \mathbb{N}$. In particular, this shows that for densities above 1/3, groups in $3k$-gonal models satisfy Property (T) with probability 1 as $n$ approaches infinity.
\end{abstract}

\maketitle

\section{Introduction}
The study of random groups is one way to approach questions of the form `Does a \textit{typical} group have property P?' To answer this question, one needs to choose a model of random groups, and then calculate the probability that property P is satisfied. Generally speaking we are interested in the asymptotic behavior of this probability as some constraint goes to infinity.

There are many models of random groups (an excellent introduction to this topic is \cite{Oll}), the most studied of which is the Gromov density model, first introduced by Gromov in \cite{Gro91}. In this model, a group presentation is chosen uniformly at random by fixing the number of generators, and choosing cyclically reduced relators of a certain length. The number of relators chosen is controlled by a fixed constant called the \textit{density}. One then analyzes the probability of a property being satisfied as both the length and the number of relators goes to infinity. This probability generally tends to 0 or 1, and in many important cases there is a \textit{sharp threshold} $d_0$, so that for densities above $d_0$ the probability is 1, and below $d_0$ the probability is 0 (or vice versa). Some examples of this phase shift behavior include properties such as hyperbolicity \cite{Gro91, Oll04}, small cancellation  \cite{Gro91}, and satisfying Dehn's algorithm \cite{Sho}.

It is known that there is a sharp threshold for Kazhdan's Property (T); however its precise value is unknown. Property (T) was introduced by Kazhdan in \cite{Kaz} as a tool to study lattices in Lie groups, but turned out to be important in several areas of mathematics and computer science, especially in the study of expander graphs (for more information, see \cite{Lub}). In \cite{Zuk}, \.{Z}uk stated and sketched a proof that for $d>1/3$, random groups in the Gromov density model satisfy, with overwhelming probability, Property (T). The proof was completed in \cite{KK}.

Similarly, we do not yet know a sharp threshold for \textit{cubulation} of random groups (i.e. admitting a cocompact action on a CAT(0) cube complex without a global fixed point). In \cite{MP}, Mackay-Przytycki built on the work of Ollivier-Wise \cite{OW} to show that at densities $d<5/24$, Gromov random groups are, with probability 1, cubulated. This bound was improved to $d<3/14$ in \cite{M} using a generalization of the methods in \cite{MP}.

For infinite hyperbolic groups, being cubulated and having Property (T) are mutually exclusive, so together these results show that the sharp threshold at which Gromov random groups fail to satisfy Property (T) must be between 1/3 and 3/14.

In order to prove his theorem, \.{Z}uk introduced a new model of random groups, called the \emph{triangular model}.  This was extended to the \emph{square model} in \cite{Odr}, and in \cite{calum} this was further generalized this to a class of models, here called the \emph{$k$-gonal models}. Roughly speaking, as $k$ approaches infinity the behavior of $k$-gonal random groups should reflect the behavior of Gromov random groups. Odrzygóźdź showed that in the $6$-gonal model, $d=1/3$ is the sharp threshold at which the probability of a group satisfying Property (T) switches from 1 to 0. Figure \ref{fig:stateoftheart} illustrates what is currently known about values of the sharp threshold for Property (T) in the $k$-gonal models. 

\begin{figure}
    \centering
    \includegraphics[width = .75\textwidth]{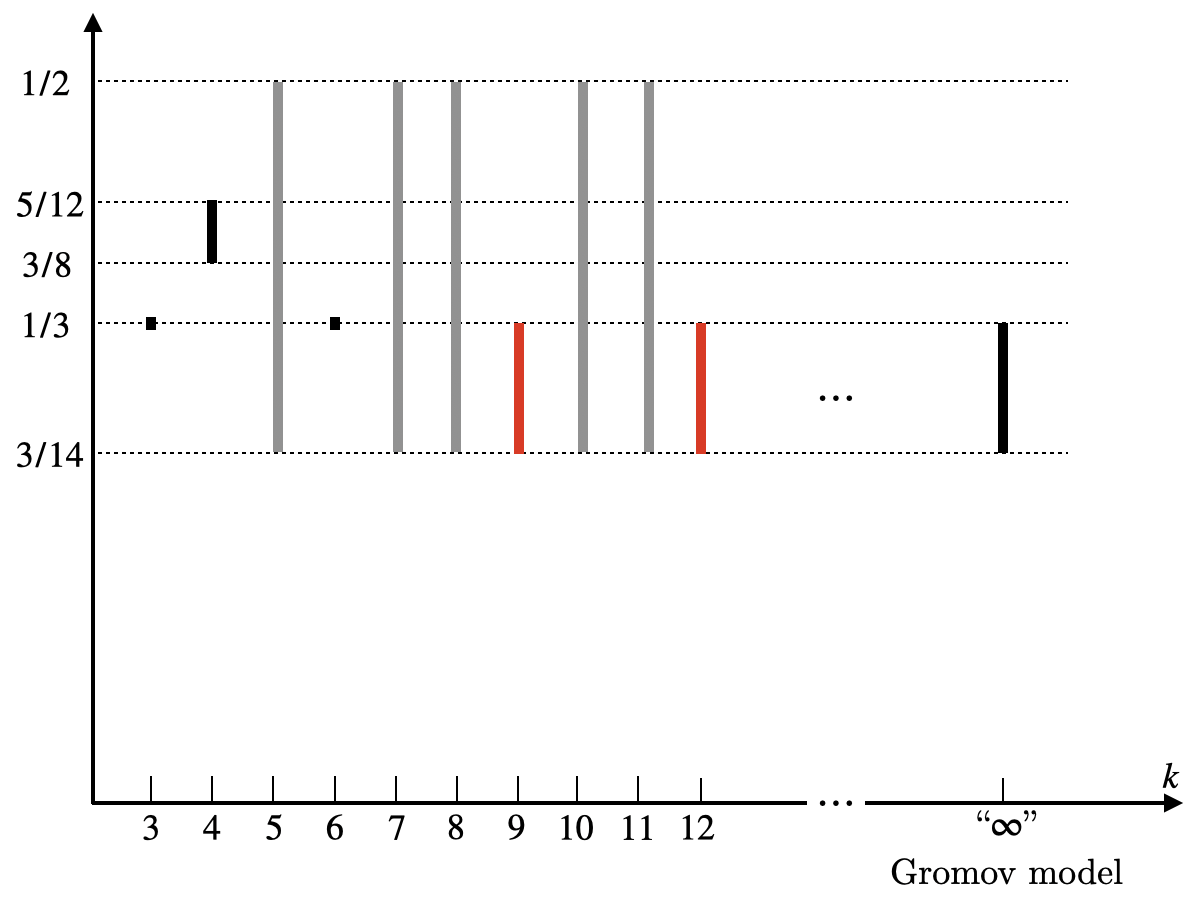}
    \caption{Possible values of the sharp threshold for Property (T) in the $k$-gonal models. Red indicates (some of the) results from this paper.}
    \label{fig:stateoftheart}
\end{figure}

In this paper, we show that show that any upper bound on the sharp threshold for Property (T) in the positive $k$-gonal model gives an upper bound in the $jk$-gonal model for all $j \in \mathbb{N}$. More precisely, we prove the following theorem:

\begin{restatable}{thm}{main}\label{thm:main}
Let $d_0$ be some fixed value in $(0, 1)$, and suppose that with overwhelming probability a random group at density $d>d_0$ in $\mathcal{M}^+_{k}(n, d)$ satisfies Property (T). Then with overwhelming probability, for any $d' \in (d_0, d)$ and for any $j \in \mathbb{N}$, a random group in $\mathcal{M}_{jk}(n, d')$ or $\mathcal{M}^+_{jk}(n, d')$ satisfies Property (T).
\end{restatable}

In particular, this improves the known upper bounds in the $3j$-gonal  models from $1/2$ to $1/3$. It also shows that to find upper bounds for the sharp threshold of Property (T), it suffices to investigate the \emph{positive} $k$-gonal models. Further, this provides evidence supporting the following conjecture:

\begin{conj}
Let $d_{(T)}(k)$ denote the sharp threshold for Property (T) in the $k$-gonal model. Then $\lim_{k \to \infty} d_{(T)}(k)$ exists, and is equal to the sharp threshold of the Gromov model.
\end{conj}

\subsection*{Acknowledgements}
I would like to thank Danny Calegari and Tomasz Odrzygóźdź for valuable comments and feedback. This research was partially supported by a Clare Boothe Luce Professorship in Mathematics from the Henry Luce Foundation.

\section{Background}
 We begin by defining the $k$-gonal model of random groups.

\begin{defn}[\cite{calum}]
Fix a set $S$ of cardinality $n$. The \emph{$k$-gonal model of random groups}, denoted $\mathcal{M}_k(n, d),$ is the set of group presentations $\langle S|R\rangle$ such that $R$ is a set of $(2n-1)^{kd}$ cyclically reduced words on $S\cup S^{-1}$ of length $k$, along with the uniform probability measure. We say that a $k$-gonal random group at density $d$ satisfies a property P \emph{with overwhelming probability}, abbreviated w.o.p, if
    \[
        \lim_{n\to\infty}\mathbb{P}(\mbox{the group defined by a presentation in }\mathcal{M}_k(n, d) \mbox{ satisfies P}) = 1.
    \]
\end{defn}

A related family of models of random groups restricts possible relators to \textit{positive} words, i.e. words on $S$ but not on $S^{-1}$. These are called the \textit{positive $k$-gonal models of random groups} and are denoted $\mathcal{M}^+_k(n, d).$

%We slightly relax the definition of $\mathcal{M}^{(+)}_k(n, d)$, and instead require that the number of relators be $ f(n) \approx (2n-1)^{kd}$, where `$\approx$' means that $\lim_{n\to \infty} f(n)/(2n-1)^{kd} = 1.$ The defintion of satisfying property $P$ w.o.p. will not depend on this detail.

We now define (Kazhdan's) Property (T) for discrete, finitely generated groups. The notation is suggestive - the idea of Property (T) is that the trivial representation of the group must be isolated in some metric. More formally, we can define this property via (almost) invariant vectors: Let $\Gamma$ be a finite generated group with finitely generating set $S$. If $\mathcal{H}$ is a Hilbert space and $\pi: \Gamma \to U(\mathcal{H})$ is a unitary representation of $\Gamma$ on $\mathcal{H}$, then $\pi$ has \emph{almost invariant vectors} if, for every $\varepsilon>0$, there exists some $u_\varepsilon \in \mathcal{H}$ such that for every $s \in S$, $\|\pi(s)u_\varepsilon - u_\varepsilon\| < \varepsilon\|u_\varepsilon\|.$ A vector $u \in \mathcal{H}$ is \emph{invariant} if $\pi(g)(u) = u$ for every $g \in \Gamma$.

\begin{defn}[\cite{Kaz}]
A finitely generated group $\Gamma$ has \emph{Property (T)} if, for every $\mathcal{H}$ and $\pi$, the following holds: if $\pi$ has almost invariant vectors, then $\pi$ has a nonzero invariant vector.
\end{defn}

The first results towards demonstrating Property (T) in random groups were first stated in \cite{Zuk}, and fully proved in \cite{KK}. The proof involved passing between several different models of random groups and analyzing spectral properties of random graphs.

\begin{thm}[\cite{KK} Theorem 3.14] \label{thm:pos triangular pt}
A random group in $\mathcal{M}^+_3(n, d')$ with density $d'>1/3$ satisfies Property (T) with overwhelming probability.
\end{thm}

In fact, they also showed that this is true in $\mathcal{M}_3(n, d)$ and in the Gromov model of random groups (at the same density).

\begin{remark}\label{rmk: surjection and finite extension}
Property (T) is preserved by surjections and finite extensions (see \cite{KK} Remark 2.10 and \cite{BdlHV} Theorem 1.7.1). In other words, if $H$ is a group with Property (T) and $H$ surjects onto $H'$, then $H'$ has Property (T); and if $H \stackrel{f.i.}{\leq} H''$ is a sub-group of finite index, then $H''$ has Property (T).
\end{remark}

The tempting argument is to say if $G \in \mathcal{M}_{jk}(n, d)$ then $G = \langle S | R \rangle$ and any relator in $R$ can be interpreted as a word of length $k$ over the alphabet of words of length $j$ on $S \cup S^{-1}$, denoted $|W_j|$. Then $G$ is in $\mathcal{M}_k(|W_j|, d).$ However, as is pointed out in \cite{KK}, problems arise around the requirement that relators be \textit{reduced}. In particular, a word of length $kj$ is a word of length $k$ in the alphabet of words of length $j$; however, a reduced word of length $k$ in the alphabet of words of length $j$ might not be reduced as a word of length $jk$. So we must first consider an intermediate group, $G^+$, whose relators are \textit{positive} words on $S$. This removes the reduction problem, and we can safely interpret this group as a random group in $\mathcal{M}_k^+(n, d)$.

The method of proof closely follows the argument used to prove  Theorem 4.3 in \cite{Odr}, which shows that groups in the $6$-gonal model w.o.p. satisfy Property (T). In particular, to prove that a random group $G \in \mathcal{M}_{jk}(n,d)$ satisfies Property (T), we construct two groups $G^+$ and $\Gamma$ so that
\[
        \Gamma \stackrel{\phi}{\twoheadrightarrow} \phi(\Gamma) \stackrel{f.i.}{\leq} G^+ \twoheadrightarrow G,
\]
and then show that $\Gamma$ satisfies Property (T).

There are two parts to this proof. The first is showing that $\phi(\Gamma)$ is finite index in $G^+$. We find a finite set of coset representatives to prove this explicitly, generalizing the method used in \cite{Odr}. The second is in showing that the group $\Gamma$ can be interpreted as a typical representative of random groups in $\mathcal{M}^+_k(m, d')$ for $d' >1/3.$ For this part of the argument, we mimic the method developed in the proof of Theorem A in \cite{KK}.

\section{Property (T) in jk-Gonal Models}
Fix a generating set $S$ of cardinality $n$. Let $W^+$ be the set of positive words on $S$ and let $W_j$ be the words of length $j$ on $S$. For any set $R$ of words on $S$, let $R^+ = R \cap W^+$. Then for any group $G = \langle S | R\rangle$ in $\mathcal{M}_{jk}(n, d)$, there is a group $G^+ = \langle S | R^+\rangle$ that surjects onto $G$.

Since $G$ is in the $jk$-gonal model, every word in $R$ and $R^+$ has length $jk$.  There is a natural injective map $\phi: W^+_j \to F(S)$, where $F(S)$ denotes the free group on $S$, given by the labeling on each word in $W^+_j$.  Define $R^+_k = \phi^{-1}(R^+),$ and notice that every word in $R^+_k$ has length $k$ over $W_j^+$. Let $\Gamma = \langle W_j^+ | R^+_k \rangle$. Then $\phi$ induces  a homomorphism, also denoted $\phi$, mapping $\Gamma \to G^+$.

Note also that $|R^+_k| = |\phi(R^+_k)| = |R^+|,$ so if $G^+ \in \mathcal{M}^+_{jk}(n, d'),$ then $\Gamma \in \mathcal{M}^+_k(|W^+_j|, d')$. Furthermore, this map takes the uniform probability distribution on $\mathcal{M}^+_{jk}(n, d')$ to the uniform probability distribution on $\mathcal{M}^+_k(|W^+_j|, d')$, because $\phi$ induces a 1-1 map on possible sets of relators.

We will show that $\phi(\Gamma)$ is a finite index subgroup of $G^+$ by finding an explicit set of coset representatives of $\phi(\Gamma)$ in $G^+$. We note that an element of $G^+$ is in $\phi(\Gamma)$ if and only if it is expressible as a string of words of length $k$, all of which are positive or the inverse thereof. We do this by inserting trivial words that allow us to group positive and negative letters into subwords of length $j$.

In the following proof, for any word $w$, $|w|_j$ denotes the word length of $w$ modulo $j$.

\begin{lemma}\label{lem:finite index}
$\phi(\Gamma)$ is a finite index subgroup of $G^+$.
\end{lemma}

\begin{proof}
Let $w$ be a word on $S \cup S^{-1}$. Without loss of generality, assume that the first letter of $w$ is in $S^{-1}$. Then $w$ can be written as $$w = A_\ell a_\ell \cdots A_2 a_2 A_1 a_1,$$ where $\{a_i\}$ are positive words on $S$, $\{A_i\}$ are negative words on $S$, and $a_1$ might be empty.

Let $x$ be a positive word of length $j - |a_1|_j$ and let $y$ be a positive word of length $j - |A_1x^{-1}|_j$.  Then we have $$w = A_\ell a_\ell \cdots A_2 a_2 y y^{-1}A_1 x^{-1}xa_1.$$ Note that $xa$ is a positive word with length divisible by $j$, and $y^{-1}A_1 x^{-1}$ is a negative word with length divisible by $j$. Therefore $w = A_\ell a_\ell \cdots A_2 a'_2 u,$ where $u \in \phi(\Gamma)$. By induction, we see that $w \in A'_\ell\phi(\Gamma)$ where $A'_\ell$ is a positive word on $S$. Then $A'_\ell = a z,$ where $j$ divides $|z|$ and $|a|< j$, so $z \in \phi(\Gamma)$. Therefore $w = azv$, where $zv \in \phi(\Gamma)$ and $|a| < j$.

Thus the (finite) set of words of length $< j$ contains a transversal of $\phi(\Gamma)$ in $G^+$, and $\phi(\Gamma)$ has finite index in $G^+$.
\end{proof}

We are now ready to prove Theorem \ref{thm:main}:

\main*

\begin{proof}
We first show that a random group $G^+$ in $\mathcal{M}^+_{kj}(n, d)$ has Property (T). Let $m = |W^+_{j}|$. By the construction above, given such a $G^+$ we can find a group $\Gamma$ in $\mathcal{M}^+_k(m, d)$ and a map $\phi: \Gamma to G^+$. The distribution induced on $\mathcal{M}^+_k(m, d)$ in this way is uniform and by Lemma \ref{lem:finite index} $\phi(\Gamma)$ is a finite index subgroup of $G^+$. Since $\Gamma$ satisfies Property (T) with overwhelming probability, by Remark \ref{rmk: surjection and finite extension}, so does $G^+$.

Now let $G =\langle S|R\rangle$ be a group in $\mathcal{M}_{kj}(n, d).$ Note that $|W^+_j| > \frac{1}{2^j}|W_j|.$ Therefore, the expected number of elements of $R^+$ is larger than $\frac{1}{2^j}|R|.$ Furthermore, for any $d'<d$ there is some $N_{d'}$ so that for $n > N_{d'},$ we have
\[
 |R^+| > \frac{1}{2^j}|R| = \frac{1}{2^j}(2n-1)^{kjd} > (2n-1)^{kjd'}.
\]

Fix $d' \in (d_0, d)$. If we enumerate the elements of $W^+_j$ and take $R'$ to be the set of the first $(2n-1)^{kjd'}$ elements of $W^+_j$ which are in $R$, we can find a group $G^+ \langle S | R'\rangle\in \mathcal{M}^+_{kj}(n, d')$ that surjects onto $G$. This gives a uniform distribution on $\mathcal{M}^+_{kj}(n, d')$, so w.o.p $G^+$ satisfies Property (T), and by Remark \ref{rmk: surjection and finite extension} so does $G$.
\end{proof}

By Theorem 3.14 in \cite{KK}, with overwhelming probability random groups in $\mathcal{M}^+_3(n, d)$ satisfy Property (T) at density $d>1/3$. By applying Theorem \ref{thm:main} we get the following result.

\begin{cor} With overwhelming probability, a random group at density $d>1/3$ in $\mathcal{M}_{3k}(n, d)$ or $\mathcal{M}^+_{3k}(n, d)$ satisfies Property (T).
\end{cor}

\printbibliography
\end{document}